\documentclass[11pt,letterpaper]{amsart}%
\usepackage{amsfonts}
\usepackage{amsmath}
\usepackage{amssymb}
\usepackage{graphicx}%
\providecommand{\U}[1]{\protect\rule{.1in}{.1in}}
\newtheorem{theorem}{Theorem}
\theoremstyle{plain}

\newtheorem{corollary}{Corollary}

\newtheorem{definition}{Definition}
\newtheorem{example}{Example}

\newtheorem{lemma}{Lemma}

\newtheorem{remark}{Remark}

\numberwithin{equation}{section}
\begin{document}
\title{Existence of periodic solutions in shifts $\delta_{\pm}$ for neutral nonlinear
dynamic systems}
\author{Murat Ad\i var}
\address[M. Ad\i var]{ Izmir University of Economics\\
Department of Mathematics, 35330, Izmir, Turkey}
\email[M. Ad\i var]{murat.adivar@ieu.edu.tr}
\urladdr{}
\author{H. Can Koyuncuo\u{g}lu}
\address[ H. C. Koyuncuo\u{g}lu]{ Izmir University of Economics\\
Department of Mathematics, 35330, Izmir, Turkey}
\email[H. C. Koyuncuo\u{g}lu]{can.koyuncuoglu@ieu.edu.tr}
\urladdr{}
\author{Youssef N. Raffoul}
\address[Y. N. Raffoul]{ University of Dayton\\
Department of Mathematics, Dayton, OH 45469-2316, USA}
\email{yraffoul1@udayton.edu}
\urladdr{}
\thanks{This study is supported by The Scientific and Technological Research Council
of Turkey}
\thanks{}
\date{January 4, 2014}
\subjclass{ Primary 34K13, 34C25, Secondary 39A13, 34N05}
\keywords{Fixed point, Floquet theory, Krasnoselskii, periodicity, Shift operators, transition matrix}

\begin{abstract}
In this study, we focus on the existence of a periodic solution for the
neutral nonlinear dynamic systems with delay%
\[
x^{\Delta}(t)=A(t)x(t)+Q^{\Delta}\left(  t,x\left(  \delta_{-}(s,t)\right)
\right)  +G\left(  t,x(t),x\left(  \delta_{-}(s,t)\right)  \right)  .
\]
We utilize the new periodicity concept in terms of shifts operators, which allows
us to extend the concept of periodicity to time scales where the additivity
requirement $t\pm T\in\mathbb{T}$ for all $t\in\mathbb{T}$ and for a fixed
$T>0,$ may not hold. More, importantly, the new concept will easily handle
time scales that are not periodic in the conventional way such as;
$\overline{q^{\mathbb{Z}}}$ and $\cup_{k=1}^{\infty}\left[  3^{\pm k},2.3^{\pm
k}\right]  \cup\left\{  0\right\}  .$ Hence, we develop a tool that enables the investigation of periodic solutions of $q$-difference systems. Since we are dealing with systems, in
order to convert our equation to an integral systems, we resort to the
transition matrix of the homogeneous Floquet system $y^{\Delta}(t)=A(t)y(t)$
and then make use of Krasnoselskii's fixed point theorem to obtain a fixed point.

\end{abstract}
\maketitle

\section{Introduction and preliminaries}

In recent decades, the theory of neutral functional equations with delays have
seen prominent attention due to its tremendous potential of its application in
applied mathematics. There are many papers that handle neutral differential
equations on regular time scales, such as discrete and continuous cases, but
few that deal with general time scales. A time scale is a nonempty arbitrary
closed subset of reals. Existence of periodic solutions is of importance to
biologists since most models deal with certain types of populations. In the
paper of Kaufmann and Raffoul \cite{Kaufman}, the authors were the first to
define the notion of periodic time scales, by requiring the additivity $t\pm
T\in\mathbb{T}$ for all $t\in\mathbb{T}$ and for a fixed $T>0,$ to hold. Of
course, as we have mentioned above, this type of requirement leaves out many
important time scales that are of interest to biologists and scientists that
\cite{Kaufman} could not handle. To overcome such difficulties, in the famous
paper of Adivar \cite{1}, the author introduced to concept of shift periodic
operators which we will utilize in our work to obtain the existence of a
periodic solution. For more on the existence of periodic solutions on regular
time scales, we refer the readers to \cite{Henriquez}, and \cite{6}. In
addition, the papers \cite{Fu&Liu}, \cite{4} and \cite{5} study the existence
of a periodic solution of system of delayed neutral functional equations by
using Sadovskii and Krasnoselskii's fixed point theorems, respectively.

Application of time scales has been extended to logistic equation modeling
population growth. We refer the reader to May \cite{May} for a detailed model
construction of
\begin{equation}
\label{log1}x^{\Delta}= -a(t)x^{\sigma}+ f(t)
\end{equation}
in the case $\mathbb{T}=\mathbb{R}$, and to \cite{elvan} for the derivation of
the equivalent time scale equation
\begin{equation}
\label{log2}x^{\Delta}= [a(t)\ominus(f(t)x)]x.
\end{equation}
More interesting application, is the version of the hematopoiesis model (Weng
and Liang \cite{Weng}),
\begin{equation}
\label{hema}x^{\Delta}(t) = -a(t)x(t) + \alpha(t)\int_{0}^{\infty}B(s)
e_{-x\beta}(t, s) \; \Delta s,
\end{equation}
where $x(t)$ is the number of red blood cells at time $t$, $\alpha, \beta,
\gamma\in C(\mathbb{T}, \mathbb{R})$ are $T$-periodic, and $B$ is a
non-negative and integrable function. This is an extension of the red cell
system on $\mathbb{R}$ introduced by Wazewska-Czyzewska and Lasota \cite{WCL}.

Throughout the paper, we assume the reader is familiar with the calculus of
time scales and for those who are interested in the theory of time scales, we
refer them to the books \cite{boh1} and \cite{boh2}.

Motivated by the papers \cite{4} and \cite{5}, we consider the nonlinear
neutral dynamic system with delay%
\[
x^{\Delta}(t)=A(t)x(t)+Q^{\Delta}\left(  t,x\left(  \delta_{-}(s,t)\right)
\right)  +G\left(  t,x(t),x\left(  \delta_{-}(s,t)\right)  \right)  ,\text{
}t\in\mathbb{T}%
\]
and by employing the results of \cite{2} and (\cite{Dacunha1}-\cite{Dacunha3})
we invert our system and then by appealing to Krasnoselskii's fixed point
theorem we will show the existence of a nonzero periodic solution by assuming
suitable conditions.

\noindent We begin by stating basic results from \cite{3} regarding shift
operators and then in the last section we focus on proving the existence of a
periodic solution using shift periodic operators. Hereafter, we use the
notation $\left[  a,b\right]  _{\mathbb{T}}$ to indicate the set $\left[
a,b\right]  \cap\mathbb{T}$. The intervals $\left[  a,b\right)  _{\mathbb{T}%
},\left(  a,b\right]  _{\mathbb{T}},$ and $\left(  a,b\right)  _{\mathbb{T}}$
are defined similarly.

\subsection{Shift operators and periodicity}

Shift operators that are periodic provide alternative tool for investigating
periodicity on time scales that may not be additive. Periodicity by means of
shift operators was first introduced in \cite{1}. In this section, we aim to
introduce basic definitions and properties of shift operators. The following
definitions, lemmas and examples can be found in \cite{1}, and \cite{3}.

\begin{definition}
\label{def3.1}Let $\mathbb{T}^{\ast}$ be a nonempty subset of the time scale
$\mathbb{T}$ including a fixed number $t_{0}\in\mathbb{T}^{\ast}$ such that
there exists operators $\delta_{\pm}:\left[  t_{0},\infty\right)
_{\mathbb{T}}\times\mathbb{T}^{\ast}\rightarrow\mathbb{T}^{\ast}$ satisfying
the following properties:
\end{definition}

\begin{enumerate}
\item The function $\delta_{\pm}$ are strictly increasing with respect to
their second arguments, if%
\[
\left(  T_{0},t\right)  ,\left(  T_{0},u\right)  \in\mathcal{D}_{\pm
}:=\left\{  \left(  s,t\right)  \in\left[  t_{0},\infty\right)  _{\mathbb{T}%
}\times\mathbb{T}^{\ast}:\delta_{\pm}\left(  s,t\right)  \in\mathbb{T}^{\ast
}\right\}  ,
\]
then%
\[
T_{0}\leq t\leq u\text{ implies }\delta_{\pm}\left(  T_{0},t\right)
\leq\delta_{\pm}\left(  T_{0},u\right)  ;
\]

\item If $\left(  T_{1},u\right)  ,\left(  T_{2},u\right)  \in\mathcal{D}_{-}$
with $T_{1}<T_{2},$ then $\delta_{-}\left(  T_{1},u\right)  >\delta_{-}\left(
T_{2},u\right)  $ and if $\left(  T_{1},u\right)  ,\left(  T_{2},u\right)
\in\mathcal{D}_{+}$ with $T_{1}<T_{2},$ then $\delta_{+}\left(  T_{1}%
,u\right)  <\delta_{+}\left(  T_{2},u\right)  ;$

\item If $t\in\left[  t_{0},\infty\right)  _{\mathbb{T}},$ then $\left(
t,t_{0}\right)  \in\mathcal{D}_{+}$ and $\delta_{+}\left(  t,t_{0}\right)
=t.$ Moreover, if $t\in\mathbb{T}^{\ast},$ then $\left(  t_{0},t\right)
\in\mathcal{D}_{+}$ and $\delta_{+}\left(  t_{0},t\right)  =t;$

\item If $\left(  s,t\right)  \in\mathcal{D}_{\pm},$ then $\left(
s,\delta_{\pm}\left(  s,t\right)  \right)  \in\mathcal{D}_{\mp}$ and
$\delta_{\mp}\left(  s,\delta_{\pm}\left(  s,t\right)  \right)  =t;$

\item If $\left(  s,t\right)  \in\mathcal{D}_{\pm}$ and $\left(  u,\delta
_{\pm}\left(  s,t\right)  \right)  \in\mathcal{D}_{\mp},$ then $\left(
s,\delta_{\mp}\left(  u,t\right)  \right)  \in\mathcal{D}_{\pm}$ and
$\delta_{\mp}\left(  u,\delta_{\pm}\left(  s,t\right)  \right)  =\delta_{\pm
}\left(  s,\delta_{\mp}\left(  u,t\right)  \right)  .$
\end{enumerate}

Then the operators $\delta_{+}$ and $\delta_{-}$ are called forward and
backward shift operators associated with the initial point $t_{0}$ on
$\mathbb{T}^{\ast}$ and the sets $\mathcal{D}_{+}$ and $\mathcal{D}_{-}$ are
domain of the operators, respectively.

\begin{example}
\label{rem 3.1}The following table shows the shift operators $\delta_{\pm
}\left(  s,t\right)  $ on some time scales:%
\[%
\begin{tabular}
[c]{|c|c|c|c|c|}\hline
$\mathbb{T}$ & $t_{0}$ & $\mathbb{T}^{\ast}$ & $\delta_{-}\left(  s,t\right)
$ & $\delta+\left(  s,t\right)  $\\\hline
$\mathbb{R}$ & $0$ & $\mathbb{R}$ & $t-s$ & $t+s$\\\hline
$\mathbb{Z}$ & $0$ & $\mathbb{Z}$ & $t-s$ & $t+s$\\\hline
$q^{\mathbb{Z}}\cup\left\{  0\right\}  $ & $1$ & $q^{\mathbb{Z}}$ & $\frac
{t}{s}$ & $st$\\\hline
$\mathbb{N}^{1/2}$ & $0$ & $\mathbb{N}^{1/2}$ & $\left(  t^{2}-s^{2}\right)
^{1/2}$ & $\left(  t^{2}+s^{2}\right)  ^{1/2}$\\\hline
\end{tabular}
\ \ .
\]

\end{example}

\begin{lemma}
\label{lem3.1}Let $\delta_{\pm}$ be the shift operators associated with the
initial point $t_{0}.$ Then we have the following:
\end{lemma}

\begin{enumerate}
\item $\delta_{-}\left(  t,t\right)  =t_{0}$ for all $t\in\left[  t_{0}%
,\infty\right)  _{\mathbb{T}};$

\item $\delta_{-}\left(  t_{0},t\right)  =t$ for all $t\in\mathbb{T}^{\ast};$

\item If $\left(  s,t\right)  \in\mathcal{D}_{+},$ then $\delta_{+}\left(
s,t\right)  =u$ implies $\delta_{-}\left(  s,u\right)  =t$ and if $\left(
s,u\right)  \in\mathcal{D}_{-},$ then $\delta_{-}\left(  s,u\right)  =t$
implies $\delta_{+}\left(  s,t\right)  =u;$

\item $\delta_{+}\left(  t,\delta_{-}\left(  s,t_{0}\right)  \right)
=\delta_{-}\left(  s,t\right)  $ for all $\left(  s,t\right)  \in
\mathcal{D}_{+}$ with $t\geq t_{0};$

\item $\delta_{+}\left(  u,t\right)  =\delta_{+}\left(  t,u\right)  $ for all
$\left(  u,t\right)  \in\left(  \left[  t_{0},\infty\right)  _{\mathbb{T}%
}\times\left[  t_{0},\infty\right)  _{\mathbb{T}}\right)  \cap\mathcal{D}%
_{+};$

\item $\delta_{+}\left(  s,t\right)  \in\left[  t_{0},\infty\right)
_{\mathbb{T}}$ for all $\left(  s,t\right)  \in\mathcal{D}_{+}$ with $t\geq
t_{0};$

\item $\delta_{-}\left(  s,t\right)  \in\left[  t_{0},\infty\right)
_{\mathbb{T}}$ for all $\left(  s,t\right)  \in\left(  \left[  t_{0}%
,\infty\right)  _{\mathbb{T}}\times\left[  s,\infty\right)  _{\mathbb{T}%
}\right)  \cap\mathcal{D}_{-};$

\item If $\delta_{+}\left(  s,.\right)  $ is $\Delta$-differentiable in its
second variable, then $\delta_{+}^{\Delta_{t}}\left(  s,.\right)  >0;$

\item $\delta_{+}\left(  \delta_{-}\left(  u,s\right)  ,\delta_{-}\left(
s,v\right)  \right)  =$ $\delta_{-}\left(  u,v\right)  $ for all $\left(
s,v\right)  \in\left(  \left[  t_{0},\infty\right)  _{\mathbb{T}}\times\left[
s,\infty\right)  _{\mathbb{T}}\right)  \cap\mathcal{D}_{-}$ \newline and
$\left(  u,s\right)  \in\left(  \left[  t_{0},\infty\right)  _{\mathbb{T}%
}\times\left[  u,\infty\right)  _{\mathbb{T}}\right)  \cap\mathcal{D}_{-};$

\item If $\left(  s,t\right)  \in$ $\mathcal{D}_{-}$ and $\delta_{-}\left(
s,t\right)  =t_{0},$ then $s=t.$
\end{enumerate}

\begin{definition}
[Periodicity in shifts]\label{def3.2} Let $\mathbb{T}$ be a time scale with
the shift operators $\delta_{\pm}$ associated with the initial point $t_{0}%
\in\mathbb{T}^{\ast},$ then $\mathbb{T}$ is said to be periodic in shifts
$\delta_{\pm},$ if there exists a $p\in(t_{0},\infty)_{\mathbb{T}^{\ast}}$
such that $\left(  p,t\right)  \in\mathcal{D}_{\mp}$ for all $t\in
\mathbb{T}^{\ast}.$ $P$ is called the period of $\ \mathbb{T}$ if
\[
P=\inf\left\{  p\in(t_{0},\infty)_{_{\mathbb{T}^{\ast}}}:\left(  p,t\right)
\in\mathcal{D}_{\mp}\text{ for all }t\in\mathbb{T}^{\ast}\right\}  >t_{0}.
\]

\end{definition}

Observe that an additive periodic time scale must be unbounded. However,
unlike additive periodic time scales a time scale, periodic in shifts, may be bounded.

\begin{example}
\label{ex new per}The following time scales are not additive periodic but
periodic in shifts $\delta_{\pm}$.

\begin{enumerate}
\item $\mathbb{T}_{1}\mathbb{=}\left\{  \pm n^{2}:n\in\mathbb{Z}\right\}  $,
$\delta_{\pm}(P,t)=\left\{
\begin{array}
[c]{ll}%
\left(  \sqrt{t}\pm\sqrt{P}\right)  ^{2} & \text{if }t>0\\
\pm P & \text{if }t=0\\
-\left(  \sqrt{-t}\pm\sqrt{P}\right)  ^{2} & \text{if }t<0
\end{array}
\right.  $, $P=1$, $t_{0}=0,$

\item $\mathbb{T}_{2}\mathbb{=}\overline{q^{\mathbb{Z}}}$, $\delta_{\pm
}(P,t)=P^{\pm1}t$, $P=q$, $t_{0}=1,$

\item $\mathbb{T}_{3}\mathbb{=}\overline{\mathbb{\cup}_{n\in\mathbb{Z}}\left[
2^{2n},2^{2n+1}\right]  }$, $\delta_{\pm}(P,t)=P^{\pm1}t$, $P=4$, $t_{0}=1,$

\item $\mathbb{T}_{4}\mathbb{=}\left\{  \frac{q^{n}}{1+q^{n}}:q>1\text{ is
constant and }n\in\mathbb{Z}\right\}  \cup\left\{  0,1\right\}  $,
\[
\delta_{\pm}(P,t)=\dfrac{q^{^{\left(  \frac{\ln\left(  \frac{t}{1-t}\right)
\pm\ln\left(  \frac{P}{1-P}\right)  }{\ln q}\right)  }}}{1+q^{\left(
\frac{\ln\left(  \frac{t}{1-t}\right)  \pm\ln\left(  \frac{P}{1-P}\right)
}{\ln q}\right)  }},\ \ P=\frac{q}{1+q}.
\]

\end{enumerate}
\end{example}

Notice that the time scale $\mathbb{T}_{4}$ in Example \ref{ex new per} is
bounded above and below and
\[
\mathbb{T}_{4}^{\ast}=\left\{  \frac{q^{n}}{1+q^{n}}:q>1\text{ is constant and
}n\in\mathbb{Z}\right\}  .
\]

\begin{corollary}
\label{Cor 1} Let $\mathbb{T}$ be a time scale that is periodic in shifts
$\delta_{\pm}$ with the period $P$. Then we have%
\begin{equation}
\delta_{\pm}(P,\sigma(t))=\sigma(\delta_{\pm}(P,t))\text{ for all }%
t\in\mathbb{T}^{\ast}\text{.} \label{sigma delta1}%
\end{equation}

\end{corollary}

\begin{example}
The time scale $\widetilde{\mathbb{T}}=(-\infty,0]\cup\lbrack1,\infty)$ cannot
be periodic in shifts $\delta_{\pm}$. Because if there was a $p\in
(t_{0},\infty)_{\widetilde{\mathbb{T}}^{\ast}}$ such that $\delta_{\pm
}(p,t)\in\widetilde{\mathbb{T}}^{\ast}$, then the point $\delta_{-}(p,0)$
would be right scattered due to (\ref{sigma delta1}). However, we have
$\delta_{-}(p,0)<0$ by (i) of Definition \ref{def3.1}. This leads to a
contradiction since every point less than $0$ is right dense.
\end{example}

\begin{definition}
[Periodic function in shifts $\delta_{\pm}$]\label{def3.3} Let $\mathbb{T}$ be
a time scale $P$-periodic in shifts. We say that a real valued function $f$
defined on $\mathbb{T}^{\ast}$ is periodic in shifts $\delta_{\pm}$ if there
exists a $T\in\left[  P,\infty\right)  _{\mathbb{T}^{\ast}}$ such that%
\begin{equation}
\left(  T,t\right)  \in\mathcal{D}_{\pm}\text{ and }f\left(  \delta_{\pm}%
^{T}\left(  t\right)  \right)  =f\left(  t\right)  \text{ for all }%
t\in\mathbb{T}^{\ast}, \label{3.1}%
\end{equation}
where $\delta_{\pm}^{T}\left(  t\right)  =\delta_{\pm}\left(  T,t\right)  $.
$T$ is called period of $f,$ if it is the smallest number satisfying
(\ref{3.1}).
\end{definition}

\begin{example}
Let $\mathbb{T=R}$ with initial point $t_{0}=1,$ the function%
\[
f\left(  t\right)  =\sin\left(  \frac{\ln\left\vert t\right\vert }{\ln\left(
1/2\right)  }\pi\right)  ,\text{ }t\in\mathbb{R}^{\ast}:=\mathbb{R-}\left\{
0\right\}
\]
is $4$-periodic in shifts $\delta_{\pm}$ since%
\begin{align*}
f\left(  \delta_{\pm}\left(  4,t\right)  \right)   &  =\left\{
\begin{array}
[c]{c}%
f\left(  t4^{\pm1}\right)  \text{ if }t\geq0\\
f\left(  t/4^{\pm1}\right)  \text{ if }t<0
\end{array}
\right. \\
&  =\sin\left(  \frac{\ln\left\vert t\right\vert \pm2\ln\left(  1/2\right)
}{\ln\left(  1/2\right)  }\pi\right) \\
&  =\sin\left(  \frac{\ln\left\vert t\right\vert }{\ln\left(  1/2\right)  }%
\pi\pm2\pi\right) \\
&  =\sin\left(  \frac{\ln\left\vert t\right\vert }{\ln\left(  1/2\right)  }%
\pi\right) \\
&  =f\left(  t\right)  .
\end{align*}

\end{example}

\begin{definition}
[$\Delta$-periodic function in shifts $\delta_{\pm}$]\label{def3.4} Let
$\mathbb{T}$ be a time scale $P$-periodic in shifts. A real valued function
$f$ defined on $\mathbb{T}^{\ast}$ is $\Delta$-periodic function in shifts if
there exists a $T\in\left[  P,\infty\right)  _{\mathbb{T}^{\ast}}$ such that%
\begin{equation}
\left(  T,t\right)  \in\mathcal{D}_{\pm}\text{ for all }t\in\mathbb{T}^{\ast}
\label{3.2}%
\end{equation}%
\begin{equation}
\text{the shifts }\delta_{\pm}^{T}\text{ are }\Delta\text{-differentiable with
rd-continuous derivatives} \label{3.3}%
\end{equation}
and%
\begin{equation}
f\left(  \delta_{\pm}^{T}\left(  t\right)  \right)  \delta_{\pm}^{\Delta
T}\left(  t\right)  =f\left(  t\right)  \label{3.4}%
\end{equation}
for all $t\in\mathbb{T}^{\ast},$ where $\delta_{\pm}^{T}\left(  t\right)
=\delta_{\pm}\left(  T,t\right)  $. The smallest number $T$ satisfying
(\ref{3.2}-\ref{3.4}) is called period of $f$.
\end{definition}

\begin{example}
The function $f\left(  t\right)  =1/t$ is $\Delta$-periodic function on
$q^{\mathbb{Z}}$ with the period $T=q$.
\end{example}

\begin{theorem}
Let $\mathbb{T}$ be a time scale that is periodic in shifts $\delta_{\pm}$
with period $P\in(t_{0},\infty)_{\mathbb{T}^{\ast}}$ and $f$ a $\Delta
$-periodic function in shifts $\delta_{\pm}$ with period $T\in\left[
P,\infty\right)  _{\mathbb{T}^{\ast}}.$ Suppose that $f\in C_{rd}%
(\mathbb{T}),$ then
\[%
{\displaystyle\int\limits_{t_{0}}^{t}}
f(s)\Delta s=%
{\displaystyle\int\limits_{\delta_{\pm}^{T}(t_{0})}^{\delta_{\pm}^{T}(t)}}
f(s)\Delta s.
\]

\end{theorem}

\subsection{Unified Floquet theory with respect to new periodicity concept}

In this section, we list some results of \cite{2} for further use.

\subsubsection{Homogeneous case}

Consider the regressive time varying linear dynamic initial value problem%
\begin{equation}
x^{\Delta}\left(  t\right)  =A\left(  t\right)  x\left(  t\right)  ,\text{
}x\left(  t_{0}\right)  =x_{0}, \label{4.1}%
\end{equation}
where $A:\mathbb{T}^{\ast}\mathbb{\rightarrow R}^{n\times n}$ is $\Delta
$-periodic in shifts with period $T$. Notice that if the time scale is
additive periodic, then $\delta_{\pm}^{\Delta}\left(  T,t\right)  =1$ and
$\Delta$-periodicity in shifts becomes the same as the periodicity in shifts.
Hence, the homogeneous system we consider in this section is more general than
the systems handled in literature.

In \cite{Dacunha3}, the solution of the system (\ref{4.1}) (for an arbitrary
matrix $A$) is expressed by the equality
\[
x\left(  t\right)  =\Phi_{A}\left(  t,t_{0}\right)  x_{0}\text{,}%
\]
where $\Phi_{A}\left(  t,t_{0}\right)  $, called the transition matrix for the
system (\ref{4.1}), is given by%
\begin{align}
\Phi_{A}\left(  t,t_{0}\right)   &  =I+%
{\displaystyle\int\limits_{t_{0}}^{t}}
A\left(  \tau_{1}\right)  \Delta\tau_{1}+%
{\displaystyle\int\limits_{t_{0}}^{t}}
A\left(  \tau_{1}\right)
{\displaystyle\int\limits_{t_{0}}^{\tau_{1}}}
A\left(  \tau_{2}\right)  \Delta\tau_{2}\Delta\tau_{1}+\ldots\nonumber\\
&  +%
{\displaystyle\int\limits_{t_{0}}^{t}}
A\left(  \tau_{1}\right)
{\displaystyle\int\limits_{t_{0}}^{\tau_{1}}}
A\left(  \tau_{2}\right)  \ldots%
{\displaystyle\int\limits_{t_{0}}^{\tau_{i-1}}}
A\left(  \tau_{i}\right)  \Delta\tau_{i}\ldots\Delta\tau_{1}+\ldots\text{.}
\label{re4.3}%
\end{align}
As mentioned in \cite{Dacunha1} the matrix exponential $e_{A}\left(
t,t_{0}\right)  $ is not always identical to $\Phi_{A}\left(  t,t_{0}\right)
$ since%
\[
A\left(  t\right)  e_{A}\left(  t,t_{0}\right)  =e_{A}\left(  t,t_{0}\right)
A\left(  t\right)
\]
is always true but the equality
\[
A\left(  t\right)  \Phi_{A}\left(  t,t_{0}\right)  =\Phi_{A}\left(
t,t_{0}\right)  A\left(  t\right)
\]
is not. It can be seen from (\ref{re4.3}) that one has $e_{A}\left(
t,t_{0}\right)  \equiv\Phi_{A}\left(  t,t_{0}\right)  $ only if the matrix $A$
satisfies%
\[
A\left(  t\right)
{\displaystyle\int\limits_{s}^{t}}
A\left(  \tau\right)  \Delta\tau=%
{\displaystyle\int\limits_{s}^{t}}
A\left(  \tau\right)  \Delta\tau A\left(  t\right)  .
\]

In preparation for the next result we define the set
\begin{equation}
P\left(  t_{0}\right)  :=\left\{  \delta_{+}^{\left(  k\right)  }\left(
T,t_{0}\right)  ,\text{ }k=0,1,2,\ldots\right\}  \label{P(t)}%
\end{equation}
and the function
\begin{equation}
\Theta\left(  t\right)  :=%
{\displaystyle\sum\limits_{j=1}^{m\left(  t\right)  }}
\delta_{-}\left(  \delta_{+}^{\left(  j-1\right)  }\left(  T,t_{0}\right)
,\delta_{+}^{\left(  j\right)  }\left(  T,t_{0}\right)  \right)  +G\left(
t\right)  , \label{4.1.1}%
\end{equation}
where
\begin{equation}
m\left(  t\right)  :=\min\left\{  k\in\mathbb{N}:\delta_{+}^{\left(  k\right)
}\left(  T,t_{0}\right)  \geq t\right\}  \label{m(t)}%
\end{equation}
and%
\begin{equation}
G\left(  t\right)  :=\left\{
\begin{array}
[c]{ll}%
0 & \text{if }t\in P\left(  t_{0}\right) \\
-\delta_{-}\left(  t,\delta_{+}^{\left(  m(t)\right)  }\left(  T,t_{0}\right)
\right)  & \text{if }t\notin P\left(  t_{0}\right)
\end{array}
\right.  . \label{G(t)}%
\end{equation}

\begin{remark}
For an additive periodic time scale we always have $\Theta\left(  t\right)
=t-t_{0}$.
\end{remark}

Following theorem constructs the matrix $R$ as a solution of matrix
exponential equation.

\begin{theorem}
[\cite{2}]\label{thm4.1} For a nonsingular, $n\times n$ constant matrix $M$ a
solution $R:\mathbb{T\rightarrow C}^{n\times n}$ of matrix exponential
equation
\[
e_{R}\left(  \delta_{+}^{T}\left(  t_{0}\right)  ,t_{0}\right)  =M
\]
can be given by%
\begin{equation}
R\left(  t\right)  =\lim_{s\rightarrow t}\frac{M^{\frac{1}{T}\left[
\Theta\left(  \sigma\left(  t\right)  \right)  -\Theta\left(  s\right)
\right]  }-I}{\sigma\left(  t\right)  -s}, \label{4.1.2}%
\end{equation}
where $I$ is the $n\times n$ identity matrix and $\Theta$ is as in
(\ref{4.1.1}).
\end{theorem}

\begin{lemma}
[\cite{2}]\label{lem4.2}Let $\mathbb{T}$ be a time scale and $P\in
\mathcal{R}\left(  \mathbb{T}^{\ast},\mathbb{R}^{n\times n}\right)  $ be a
$\Delta-$periodic matrix valued function in shifts with period $T$, i.e.%
\[
P\left(  t\right)  =P\left(  \delta_{\pm}^{T}\left(  t\right)  \right)
\delta_{\pm}^{\Delta T}\left(  t\right)
\]
Then the solution of the dynamic matrix initial value problem%
\begin{equation}
Y^{\Delta}\left(  t\right)  =P\left(  t\right)  Y\left(  t\right)  ,\text{
}Y\left(  t_{0}\right)  =Y_{0}, \label{2}%
\end{equation}
is unique up to a period $T$ in shifts. That is
\begin{equation}
\Phi_{P}\left(  t,t_{0}\right)  =\Phi_{P}\left(  \delta_{+}^{T}\left(
t\right)  ,\delta_{+}^{T}\left(  t_{0}\right)  \right)  \label{2.1}%
\end{equation}
for all $t\in\mathbb{T}^{\ast}$.
\end{lemma}

\begin{corollary}
[\cite{2}]Let $\mathbb{T}$ be a time scale and $P\in\mathcal{R}\left(
\mathbb{T}^{\ast},\mathbb{R}^{n\times n}\right)  $ be a $\Delta-$periodic
matrix valued function in shifts, i.e.%
\[
P\left(  t\right)  =P\left(  \delta_{\pm}^{T}\left(  t\right)  \right)
\delta_{\pm}^{\Delta T}\left(  t\right)
\]
Then%
\begin{equation}
e_{P}\left(  t,t_{0}\right)  =e_{P}\left(  \delta_{+}^{T}\left(  t\right)
,\delta_{+}^{T}\left(  t_{0}\right)  \right)  . \label{2.2}%
\end{equation}

\end{corollary}

\begin{theorem}
[\cite{2},Floquet decomposition]\label{thm1} Let $A$ be a matrix valued
function that is $\Delta$-periodic in shifts with period $T$. The transition
matrix for $A$ can be given in the form%
\begin{equation}
\Phi_{A}\left(  t,\tau\right)  =L\left(  t\right)  e_{R}\left(  t,\tau\right)
L^{-1}\left(  \tau\right)  ,\text{ for all }t,\tau\in\mathbb{T}^{\ast},
\label{3}%
\end{equation}
where $R:\mathbb{T\rightarrow C}^{n\times n}$ and $L\left(  t\right)  \in
C_{rd}^{1}\left(  \mathbb{T}^{\ast},\mathbb{R}^{n\times n}\right)  $ are both
periodic in shifts with period $T$ and invertible.
\end{theorem}

\begin{theorem}
[\cite{2}]\label{thm3}There exists an initial state $x\left(  t_{0}\right)
=x_{0}\neq0$ such that the solution of (\ref{4.1}) is $T$-periodic in shifts
if and only if one of the eigenvalues of the matrix%
\[
e_{R}\left(  \delta_{+}^{T}\left(  t_{0}\right)  ,t_{0}\right)  =\Phi
_{A}\left(  \delta_{+}^{T}\left(  t_{0}\right)  ,t_{0}\right)
\]
is $1$.
\end{theorem}

\subsubsection{Nonhomogeneous case}

Let us consider the nonhomogeneous regressive nonautonomous linear dynamic
initial value problem%
\begin{equation}
x^{\Delta}\left(  t\right)  =A\left(  t\right)  x\left(  t\right)  +F\left(
t\right)  ,\text{ }x\left(  t_{0}\right)  =x_{0}, \label{6}%
\end{equation}
where $A:\mathbb{T}^{\ast}\mathbb{\rightarrow R}^{n\times n}$,$\ F\in
C_{rd}\left(  \mathbb{T}^{\ast},\mathbb{R}^{n}\right)  \cap\mathcal{R}\left(
\mathbb{T}^{\ast},\mathbb{R}^{n}\right)  $. Hereafter, we suppose both $A$ and
$F$ are $\Delta$-periodic in shifts with the period $T$.

\begin{theorem}
[\cite{2}]\label{thm5} For any initial point $t_{0}\in\mathbb{T}^{\ast}$ and
for any function $F$ that is $\Delta$-periodic in shifts with period $T$,
there exists an initial state $x\left(  t_{0}\right)  =x_{0}$ such that the
solution of (\ref{6}) is $T$-periodic in shifts if and only if there does not
exist a nonzero $z\left(  t_{0}\right)  =z_{0}$ and $t_{0}\in\mathbb{T}^{\ast
}$ such that the $T$-periodic homogeneous initial value problem%
\begin{equation}
z^{\Delta}\left(  t\right)  =A\left(  t\right)  z\left(  t\right)  ,\text{
}z\left(  t_{0}\right)  =z_{0}, \label{8}%
\end{equation}
has a solution that is $T$-periodic in shifts.
\end{theorem}

For details about Floquet theory based on new periodicity concept on time
scales, we refer readers \cite{2}.

\section{Existence of periodic solutions}

For $T>0,$ let $P_{T}$ be the set of all $n$-vector functions $x(t),$ periodic
in shifts with period $T.$ Then $\left(  P_{T},\left\Vert .\right\Vert
\right)  $ is a Banach space endowed with the norm%
\[
\left\Vert x\right\Vert =\max_{t\in\left[  t_{0},\delta_{+}^{T}(t_{0})\right]
_{\mathbb{T}}}\left\vert x(t)\right\vert .
\]
Also for an $n\times n$ matrix valued function $A$, given by $A\left(
t\right)  :=\left[  a_{ij}\left(  t\right)  \right]  $, we define the norm
$\left\Vert A\right\Vert $ by
\[
\left\Vert A\right\Vert =\sup_{t\in\lbrack t_{0},\infty)_{\mathbb{T}}%
}\left\vert A\left(  t\right)  \right\vert ,
\]
where%
\[
\left\vert A(t)\right\vert :=\max_{1\leq i\leq n}%
{\displaystyle\sum\limits_{j=1}^{n}}
\left\vert a_{ij}\left(  t\right)  \right\vert .
\]

Now assume that $\mathbb{T}$ is a time scale that is $T$-periodic in shifts
and consider the delay dynamic system%
\begin{equation}
x^{\Delta}(t)=A(t)x(t)+Q^{\Delta}\left(  t,x\left(  \delta_{-}(s,t)\right)
\right)  +G\left(  t,x(t),x\left(  \delta_{-}(s,t)\right)  \right)  ,
\label{5.1}%
\end{equation}
where $A\in C_{rd}(\mathbb{T}^{\ast},\mathbb{R}^{n\times n}),$ $Q\in
C_{rd}\left(  \mathbb{T}^{\ast}\times\mathbb{T}^{\ast n},\mathbb{R}%
^{n}\right)  \mathbb{\ }$and $G\in C_{rd}\left(  \mathbb{T}^{\ast}%
\times\mathbb{T}^{\ast n}\times\mathbb{T}^{\ast n},\mathbb{R}^{n}\right)  .$
Since we focus on the existence of a periodic solution of (\ref{5.1}), we have
the following periodicity assumptions:%
\begin{equation}
A\text{ is }\Delta\text{-periodic in shifts, i.e. }A(\delta_{\pm}^{T}\left(
t\right)  )\delta_{\pm}^{\Delta T}\left(  t\right)  =A(t)\text{ for all }%
t\in\mathbb{T}^{\ast}, \label{5.2}%
\end{equation}%
\begin{equation}
Q(\delta_{\pm}^{T}\left(  t\right)  ,x\left(  \delta_{-}(s,\delta_{\pm}%
^{T}\left(  t\right)  \right)  )=Q\left(  t,x\left(  \delta_{-}(s,t)\right)
\right)  \text{ for all }t\in\mathbb{T}^{\ast}\text{ and }x\in P_{T},
\label{5.3}%
\end{equation}
and%
\begin{equation}
G\left(  \delta_{\pm}^{T}\left(  t\right)  ,x(\delta_{\pm}^{T}\left(
t\right)  ),x\left(  \delta_{-}(s,\delta_{\pm}^{T}\left(  t\right)  )\right)
\right)  \delta_{\pm}^{\Delta T}\left(  t\right)  =G\left(  t,x(t),x\left(
\delta_{-}(s,t)\right)  \right)  \text{ for all }t\in\mathbb{T}^{\ast}\text{
and }x\in P_{T}. \label{5.4}%
\end{equation}
In order to prove the existence of a nonzero periodic solution in shifts
$\delta_{\pm}$ for system (\ref{5.1}), we assume that%
\begin{equation}
Q^{\Delta}\left(  t,0\right)  +G\left(  t,0,0\right)  \neq0 \label{5.4.1}%
\end{equation}
for some $t\in\mathbb{T}^{\ast}$.

Throughout the paper, we assume that the homogeneous system
\begin{equation}
z^{\Delta}\left(  t\right)  =A\left(  t\right)  z\left(  t\right)  ,\text{
}z\left(  t_{0}\right)  =z_{0} \label{5.7}%
\end{equation}
is non-critical. That is (\ref{5.7}) has no periodic nonzero solution in
shifts $\delta_{\pm}.$

\begin{lemma}
\label{lem5.1}Suppose that (\ref{5.2}-\ref{5.4.1}) hold. If $x(t)\in P_{T},$
then $x(t)$ is a solution of equation (\ref{5.1}) satisfying $x(t_{0})=x_{0}$
if and only if%
\begin{align}
x(t)  &  =Q(t,x\left(  \delta_{-}(s,t)\right)  )+\nonumber\\
&  \Phi_{A}(t,t_{0})\left(  \Phi_{A}^{-1}(\delta_{+}^{T}(t_{0}),t_{0}%
)-I\right)  ^{-1}\times\nonumber\\
&
{\displaystyle\int_{t}^{\delta_{+}^{T}(t)}}
\Phi_{A}^{-1}(\sigma(u),t_{0})\left[  A(u)Q(u,x(\delta_{-}%
(s,u)))+G(u,x(u),x(\delta_{-}(s,u)))\right]  \Delta u. \label{5.8}%
\end{align}

\end{lemma}

\begin{proof}
Let $x(t)\in P_{T}$ be a solution of (\ref{5.1}) satisfying $x(t_{0})=x_{0}$
and $\Phi_{A}(t,t_{0})$ be the transition matrix of system (\ref{5.7}). The
necessity part of the proof is straightforward. For the sufficiency part we
employ (\ref{5.1}) to get%
\begin{align*}
\left[  x(t)-Q(t,x\left(  \delta_{-}(s,t)\right)  )\right]  ^{\Delta}  &
=A(t)\left(  x(t)-Q(t,x\left(  \delta_{-}(s,t)\right)  )\right)
+A(t)Q(t,x\left(  \delta_{-}(s,t)\right)  )\\
&  +G\left(  t,x(t),x\left(  \delta_{-}(s,t)\right)  \right)  .
\end{align*}
Since $\Phi_{A}(t,t_{0})\Phi_{A}^{-1}(t,t_{0})=I,$ we have%
\begin{align*}
0  &  =\left(  \Phi_{A}(t,t_{0})\Phi_{A}^{-1}(t,t_{0})\right)  ^{\Delta}\\
&  =\Phi_{A}^{\Delta}(t,t_{0})\Phi_{A}^{-1}(t,t_{0})+\Phi_{A}(\sigma
(t),t_{0})\left(  \Phi_{A}^{-1}(t,t_{0})\right)  ^{\Delta}\\
&  =\left(  A(t)\Phi_{A}(t,t_{0})\right)  \Phi_{A}^{-1}(t,t_{0})+\Phi
_{A}(\sigma(t),t_{0})\left(  \Phi_{A}^{-1}(t,t_{0})\right)  ^{\Delta}\\
&  =A(t)+\Phi_{A}(\sigma(t),t_{0})\left(  \Phi_{A}^{-1}(t,t_{0})\right)
^{\Delta}.
\end{align*}
That is,%
\begin{equation}
\left(  \Phi_{A}^{-1}(t,t_{0})\right)  ^{\Delta}=-\Phi_{A}^{-1}(\sigma
(t),t_{0})A(t). \label{5.9}%
\end{equation}
If $x(t)$ is a solution of (\ref{5.1}) satisfying $x(t_{0})=x_{0},$ then
\begin{align*}
\left\{  \Phi_{A}^{-1}(t,t_{0})\left(  x(t)-Q(t,x\left(  \delta_{-}%
(s,t)\right)  )\right)  \right\}  ^{\Delta}  &  =\left(  \Phi_{A}^{-1}%
(t,t_{0})\right)  ^{\Delta}\left(  x(t)-Q(t,x\left(  \delta_{-}(s,t)\right)
)\right) \\
&  +\Phi_{A}^{-1}(\sigma(t),t_{0})\left(  x(t)-Q(t,x\left(  \delta
_{-}(s,t)\right)  )\right)  ^{\Delta}\\
&  =-\Phi_{A}^{-1}(\sigma(t),t_{0})A(t)\left(  x(t)-Q(t,x\left(  \delta
_{-}(s,t)\right)  )\right) \\
&  +\Phi_{A}^{-1}(\sigma(t),t_{0})\left[  A(t)\left(  x(t)-Q(t,x\left(
\delta_{-}(s,t)\right)  )\right)  \right. \\
&  +\left.  A(t)Q(t,x\left(  \delta_{-}(s,t)\right)  )+G\left(
t,x(t),x\left(  \delta_{-}(s,t)\right)  \right)  \right] \\
&  =\Phi_{A}^{-1}(\sigma(t),t_{0})\left[  A(t)Q(t,x\left(  \delta
_{-}(s,t)\right)  )+G\left(  t,x(t),x\left(  \delta_{-}(s,t)\right)  \right)
\right]  .
\end{align*}
Integrating the last equality from $t_{0}$ to $t,$ we arrive at%
\begin{align}
x(t)  &  =Q(t,x\left(  \delta_{-}(s,t)\right)  )+\Phi_{A}(t,t_{0})\left(
x_{0}-Q(t_{0},x\left(  \delta_{-}(s,t_{0})\right)  )\right) \nonumber\\
&  +\Phi_{A}(t,t_{0})%
{\displaystyle\int\limits_{t_{0}}^{t}}
\Phi_{A}^{-1}(\sigma(u),t_{0})\left[  A(u)Q(u,x\left(  \delta_{-}(s,u)\right)
)+G\left(  u,x(u),x\left(  \delta_{-}(s,u)\right)  \right)  \right]  \Delta u.
\label{5.10}%
\end{align}
Since $x(\delta_{+}^{T}(t_{0}))=x(t_{0})=x_{0},$ (\ref{5.10}) implies%
\begin{align}
x(t_{0})-Q(t_{0},x(\delta_{-}(s,t_{0})))  &  =\Phi_{A}(\delta_{+}^{T}%
(t_{0}),t_{0})\left(  x_{0}-Q(t_{0},x\left(  \delta_{-}(s,t_{0})\right)
\right)  )\nonumber\\
&  +\Phi_{A}(\delta_{+}^{T}(t_{0}),t_{0})%
{\displaystyle\int\limits_{t_{0}}^{\delta_{+}^{T}(t_{0})}}
\Phi_{A}^{-1}(\sigma(u),t_{0})\left[  A(u)Q(u,x\left(  \delta_{-}(s,u)\right)
)\right. \label{5.11}\\
&  +\left.  G\left(  u,x(u),x\left(  \delta_{-}(s,u)\right)  \right)  \right]
\Delta u.\nonumber
\end{align}
Substituting (\ref{5.11}) into (\ref{5.10}) yields%
\begin{align}
x(t)  &  =Q(t,x\left(  \delta_{-}(s,t)\right)  )+\Phi_{A}(t,t_{0})\left(
I-\Phi_{A}(\delta_{+}^{T}(t_{0}),t_{0})\right)  ^{-1}\Phi_{A}(\delta_{+}%
^{T}(t_{0}),t_{0})\nonumber\\
&  \times%
{\displaystyle\int\limits_{t_{0}}^{\delta_{+}^{T}(t_{0})}}
\Phi_{A}^{-1}(\sigma(u),t_{0})\left[  A(u)Q(u,x\left(  \delta_{-}(s,u)\right)
)+G\left(  u,x(u),x\left(  \delta_{-}(s,u)\right)  \right)  \right]  \Delta
u\text{ }\nonumber\\
&  +\Phi_{A}(t,t_{0})%
{\displaystyle\int\limits_{t_{0}}^{t}}
\Phi_{A}^{-1}(\sigma(u),t_{0})\left[  A(u)Q(u,x\left(  \delta_{-}(s,u)\right)
)+G\left(  u,x(u),x\left(  \delta_{-}(s,u)\right)  \right)  \right]  \Delta
u.\text{ } \label{5.12}%
\end{align}
In order to show that (\ref{5.12}) is equivalent to (\ref{5.8}) we use
\begin{align*}
\left(  I-\Phi_{A}(\delta_{+}^{T}(t_{0}),t_{0})\right)  ^{-1}  &  =\left(
\Phi_{A}(\delta_{+}^{T}(t_{0}),t_{0})\left(  \Phi_{A}^{-1}(\delta_{+}%
^{T}(t_{0}),t_{0})-I\right)  \right)  ^{-1}\\
&  =\left(  \Phi_{A}^{-1}(\delta_{+}^{T}(t_{0}),t_{0})-I\right)  ^{-1}\Phi
_{A}^{-1}(\delta_{+}^{T}(t_{0}),t_{0}).
\end{align*}
to get%
\begin{align*}
x(t)  &  =Q(t,x\left(  \delta_{-}(s,t)\right)  )+\Phi_{A}(t,t_{0})\left(
\Phi_{A}^{-1}(\delta_{+}^{T}(t_{0}),t_{0})-I\right)  ^{-1}\Phi_{A}^{-1}%
(\delta_{+}^{T}(t_{0}),t_{0})\Phi_{A}(\delta_{+}^{T}(t_{0}),t_{0})\\
&  \times%
{\displaystyle\int\limits_{t_{0}}^{\delta_{+}^{T}(t_{0})}}
\Phi_{A}^{-1}(\sigma(u),t_{0})\left[  A(u)Q(u,x\left(  \delta_{-}(s,u)\right)
)+G\left(  u,x(u),x\left(  \delta_{-}(s,u)\right)  \right)  \right]  \Delta
u\\
&  +\Phi_{A}(t,t_{0})%
{\displaystyle\int\limits_{t_{0}}^{t}}
\Phi_{A}^{-1}(\sigma(u),t_{0})\left[  A(u)Q(u,x\left(  \delta_{-}(s,u)\right)
)+G\left(  u,x(u),x\left(  \delta_{-}(s,u)\right)  \right)  \right]  \Delta u,
\end{align*}
and we have the following equality:%
\begin{align*}
x(t)  &  =Q(t,x\left(  \delta_{-}(s,t)\right)  )+\Phi_{A}(t,t_{0})\left(
\Phi_{A}^{-1}(\delta_{+}^{T}(t_{0}),t_{0})-I\right)  ^{-1}\\
&  \times\left[
{\displaystyle\int\limits_{t_{0}}^{\delta_{+}^{T}(t_{0})}}
\Phi_{A}^{-1}(\sigma(u),t_{0})\left[  A(u)Q(u,x\left(  \delta_{-}(s,u)\right)
)+G\left(  u,x(u),x\left(  \delta_{-}(s,u)\right)  \right)  \right]  \Delta
u\right. \\
&  +\left.  \Phi_{A}^{-1}(\delta_{+}^{T}(t_{0}),t_{0})%
{\displaystyle\int\limits_{t_{0}}^{t}}
\Phi_{A}^{-1}(\sigma(u),t_{0})\left[  A(u)Q(u,x\left(  \delta_{-}(s,u)\right)
)+G\left(  u,x(u),x\left(  \delta_{-}(s,u)\right)  \right)  \right]  \Delta
u\right. \\
&  -\left.
{\displaystyle\int\limits_{t_{0}}^{t}}
\Phi_{A}^{-1}(\sigma(u),t_{0})\left[  A(u)Q(u,x\left(  \delta_{-}(s,u)\right)
)+G\left(  u,x(u),x\left(  \delta_{-}(s,u)\right)  \right)  \right]  \Delta
u\right]  .
\end{align*}
Thus, $x(t)$ can be stated as follows
\begin{align*}
x(t)  &  =Q(t,x\left(  \delta_{-}(s,t)\right)  )+\Phi_{A}(t,t_{0})\left(
\Phi_{A}^{-1}(\delta_{+}^{T}(t_{0}),t_{0})-I\right)  ^{-1}\\
&  \times\left[
{\displaystyle\int\limits_{t}^{\delta_{+}^{T}(t_{0})}}
\Phi_{A}^{-1}(\sigma(u),t_{0})\left[  A(u)Q(u,x\left(  \delta_{-}(s,u)\right)
)+G\left(  u,x(u),x\left(  \delta_{-}(s,u)\right)  \right)  \right]  \Delta
u\right. \\
&  +\left.  \Phi_{A}^{-1}(\delta_{+}^{T}(t_{0}),t_{0})%
{\displaystyle\int\limits_{t_{0}}^{t}}
\Phi_{A}^{-1}(\sigma(u),t_{0})\left[  A(u)Q(u,x\left(  \delta_{-}(s,u)\right)
)+G\left(  u,x(u),x\left(  \delta_{-}(s,u)\right)  \right)  \right]  \Delta
u\right]  .
\end{align*}
If we let $u=\delta_{-}^{T}(\hat{u}),$ we get%
\begin{align*}
x(t)  &  =Q(t,x\left(  \delta_{-}(s,t)\right)  )+\Phi_{A}(t,t_{0})\left(
\Phi_{A}^{-1}(\delta_{+}^{T}(t_{0}),t_{0})-I\right)  ^{-1}\\
&  \times\left[
{\displaystyle\int\limits_{t}^{\delta_{+}^{T}(t_{0})}}
\Phi_{A}^{-1}(\sigma(u),t_{0})\left[  A(u)Q(u,x\left(  \delta_{-}(s,u)\right)
)+G\left(  u,x(u),x\left(  \delta_{-}(s,u)\right)  \right)  \right]  \Delta
u\right. \\
&  +\left.  \Phi_{A}^{-1}(\delta_{+}^{T}(t_{0}),t_{0})\left[
{\displaystyle\int\limits_{\delta_{+}^{T}(t_{0})}^{\delta_{+}^{T}(t)}}
\Phi_{A}^{-1}(\sigma(\delta_{-}^{T}(\hat{u})),t_{0})\left[  A(\delta_{-}%
^{T}(\hat{u}))Q(\delta_{-}^{T}(\hat{u}),x\left(  \delta_{-}(s,\delta_{-}%
^{T}(\hat{u}))\right)  )\right.  \right.  \right. \\
&  +\left.  \left.  \left[  G\left(  \delta_{-}^{T}(\hat{u}),x(\delta_{-}%
^{T}(\hat{u})),x\left(  \delta_{-}(s,\delta_{-}^{T}(\hat{u}))\right)  \right)
\right]  \delta_{-}^{\Delta T}(\hat{u})\Delta u\right]  \right]
\end{align*}
and%
\begin{align*}
x(t)  &  =Q(t,x\left(  \delta_{-}(s,t)\right)  )+\Phi_{A}(t,t_{0})\left(
\Phi_{A}^{-1}(\delta_{+}^{T}(t_{0}),t_{0})-I\right)  ^{-1}\\
&  \times\left[
{\displaystyle\int\limits_{t}^{\delta_{+}^{T}(t_{0})}}
\Phi_{A}^{-1}(\sigma(u),t_{0})\left[  A(u)Q(u,x\left(  \delta_{-}(s,u)\right)
)+G\left(  u,x(u),x\left(  \delta_{-}(s,u)\right)  \right)  \right]  \Delta
u\right. \\
&  +\left.  \Phi_{A}^{-1}(\delta_{+}^{T}(t_{0}),t_{0})%
{\displaystyle\int\limits_{\delta_{+}^{T}(t_{0})}^{\delta_{+}^{T}(t)}}
\Phi_{A}^{-1}(\sigma(\delta_{-}^{T}(\hat{u})),t_{0})\left[  A(\hat{u}%
)Q(\hat{u},x\left(  \delta_{-}(s,\hat{u})\right)  )+G\left(  \hat{u},x(\hat
{u}),x\left(  \delta_{-}(s,\hat{u})\right)  \right)  \right]  \Delta u\right]
.
\end{align*}
Since%
\[
\Phi_{A}^{-1}(\delta_{+}^{T}(t_{0}),t_{0})\Phi_{A}^{-1}(\sigma(\delta_{-}%
^{T}(\hat{u})),t_{0})=\Phi_{A}^{-1}(\delta_{+}^{T}(t_{0}),t_{0})\Phi_{A}%
^{-1}\left(  \delta_{-}^{T}(\sigma(\hat{u})),t_{0}\right)  ,
\]
we have%
\begin{align*}
\Phi_{A}(t_{0},\delta_{+}^{T}(t_{0}))\Phi_{A}\left(  t_{0},\delta_{-}%
^{T}(\sigma(\hat{u}))\right)   &  =\Phi_{A}(t_{0},\delta_{+}^{T}(t_{0}%
))\Phi_{A}\left(  \delta_{+}^{T}(t_{0}),\sigma(\hat{u})\right) \\
&  =\Phi_{A}(t_{0},\sigma(\hat{u})).
\end{align*}
Applying the last equality to the preceding one, we obtain%
\begin{align*}
x(t)  &  =Q(t,x\left(  \delta_{-}(s,t)\right)  )+\Phi_{A}(t,t_{0})\left(
\Phi_{A}^{-1}(\delta_{+}^{T}(t_{0}),t_{0})-I\right)  ^{-1}\\
&  \times\left[
{\displaystyle\int\limits_{t}^{\delta_{+}^{T}(t)}}
\Phi_{A}^{-1}(\sigma(u),t_{0})\left[  A(u)Q(u,x\left(  \delta_{-}(s,u)\right)
)+G\left(  u,x(u),x\left(  \delta_{-}(s,u)\right)  \right)  \right]  \Delta
u\right]  ,
\end{align*}
as desired.
\end{proof}

Next we state Krasnoselskii's fixed point theorem which we employ for showing
existence of a periodic solution.

\begin{theorem}
[Krasnoselskii]\label{fixpoint}Let $\mathbb{M}$ be a closed convex nonempty
subset of a Banach space $\left(  \mathbb{B},\left\Vert .\right\Vert \right)
.$ Suppose that $B$ and $C$ maps $\mathbb{M}$ into $\mathbb{B}$ such that

\begin{enumerate}
\item[(i)] $x,y\in\mathbb{M}$, implies $Bx+Cy\in\mathbb{M}$,

\item[(ii)] $C$ is compact and continuous,

\item[(iii)] $B$ is a contraction mapping.
\end{enumerate}

Then there exists $z\in\mathbb{M}$ with $z=Bz+Cz.$
\end{theorem}

In preparation for the next result define the mapping $H$ by
\begin{equation}
\left(  H\varphi\right)  \left(  t\right)  =\left(  B\varphi\right)  \left(
t\right)  +\left(  C\varphi\right)  \left(  t\right)  , \label{H}%
\end{equation}
where
\begin{equation}
\left(  B\varphi\right)  (t):=Q(t,\varphi\left(  \delta_{-}(s,t)\right)  )
\label{5.13}%
\end{equation}
and%
\begin{align}
\left(  C\varphi\right)  (t)  &  :=\Phi_{A}(t,t_{0})\left(  \Phi_{A}%
^{-1}(\delta_{+}^{T}(t_{0}),t_{0})-I\right)  ^{-1}\nonumber\\
&  \times%
{\displaystyle\int\limits_{t}^{\delta_{+}^{T}(t)}}
\Phi_{A}^{-1}(\sigma(u),t_{0})\left[  A(u)Q(u,\varphi\left(  \delta
_{-}(s,u)\right)  )+G\left(  u,\varphi(u),\varphi\left(  \delta_{-}%
(s,u)\right)  \right)  \right]  \Delta u. \label{5.14}%
\end{align}

\begin{lemma}
\label{c-c}Suppose (\ref{5.2})-(\ref{5.4.1}) hold. Let $C$ be defined by
(\ref{5.14}). If there exist positive constants $E_{1},E_{2}$, $E_{3}$, and
$N$ such that%
\begin{equation}
\left\vert Q(t,x)-Q(t,y)\right\vert \leq E_{1}\left\Vert x-y\right\Vert ,
\label{5.5}%
\end{equation}%
\begin{equation}
\left\vert G(t,x,y)-G(t,z,w)\right\vert \leq E_{2}\left\Vert x-z\right\Vert
+E_{3}\left\Vert y-w\right\Vert , \label{5.6}%
\end{equation}
and%
\begin{equation}
r\left(  \delta_{+}^{T}(t_{0})-t_{0}\right)  \left(  \left\Vert A\right\Vert
E_{1}+E_{2}+E_{3}\right)  \leq N \label{N}%
\end{equation}
hold, then

\begin{enumerate}
\item[(i)] \label{part1}%
\[
\left\Vert \left(  C\varphi\right)  (.)\right\Vert \leq r\left(  \delta
_{+}^{T}(t_{0})-t_{0}\right)  \left\Vert A(.)Q(.,\varphi\left(  \delta
_{-}(s,.)\right)  )+G\left(  .,\varphi(.),\varphi\left(  \delta_{-}%
(s,.)\right)  \right)  \right\Vert ,
\]
where%
\begin{equation}
r=\max_{t\in\left[  t_{0},\delta_{+}^{T}(t_{0})\right]  _{\mathbb{T}}}\left(
\max_{u\in\left[  t,\delta_{+}^{T}(t)\right]  _{\mathbb{T}}}\left\vert \left[
\Phi_{A}(\sigma(u),t_{0})\left(  \Phi_{A}^{-1}(\delta_{+}^{T}(t_{0}%
),t_{0})-I\right)  \Phi_{A}^{-1}(t,t_{0})\right]  ^{-1}\right\vert \right)  .
\label{r}%
\end{equation}

\item[(ii)] $C$ is continuous and compact.
\end{enumerate}
\end{lemma}

\begin{proof}
Let $C$ be defined as in (\ref{5.14}). Then it can be written as in the
following form%
\begin{align*}
\left(  C\varphi\right)  (t)  &  =%
{\displaystyle\int\limits_{t}^{\delta_{+}^{T}(t)}}
\left[  \left[  \Phi_{A}(\sigma(u),t_{0})\left(  \Phi_{A}^{-1}(\delta_{+}%
^{T}(t_{0}),t_{0})-I\right)  \Phi_{A}^{-1}(t,t_{0})\right]  ^{-1}\right. \\
&  \times\left.  \left[  A(u)Q(u,\varphi\left(  \delta_{-}(s,u)\right)
)+G\left(  u,\varphi(u),\varphi\left(  \delta_{-}(s,u)\right)  \right)
\right]  \right]  \Delta u.
\end{align*}
Since $\left(  C\varphi\right)  (t)\in P_{T},$ we have%
\begin{align*}
\left\Vert \left(  C\varphi\right)  (.)\right\Vert  &  =\max_{t\in\left[
t_{0},\delta_{+}^{T}(t_{0})\right]  _{\mathbb{T}}}\left\vert
{\displaystyle\int\limits_{t}^{\delta_{+}^{T}(t)}}
\left[  \Phi_{A}(\sigma(u),t_{0})\left(  \Phi_{A}^{-1}(\delta_{+}^{T}%
(t_{0}),t_{0})-I\right)  \Phi_{A}^{-1}(t,t_{0})\right]  ^{-1}\left[
A(u)Q(u,\varphi\left(  \delta_{-}(s,u)\right)  )\right.  \right. \\
&  \left.  \left.  +G\left(  u,\varphi(u),\varphi\left(  \delta_{-}%
(s,u)\right)  \right)  \right]  \Delta u\right\vert \\
&  \leq\max_{t\in\left[  t_{0},\delta_{+}^{T}(t_{0})\right]  _{\mathbb{T}}%
}\left(  \max_{u\in\left[  t,\delta_{+}^{T}(t)\right]  _{\mathbb{T}}%
}\left\vert \left[  \Phi_{A}(\sigma(u),t_{0})\left(  \Phi_{A}^{-1}(\delta
_{+}^{T}(t_{0}),t_{0})-I\right)  \Phi_{A}^{-1}(t,t_{0})\right]  ^{-1}%
\right\vert \right) \\
&  \times\max_{t\in\left[  t_{0},\delta_{+}^{T}(t_{0})\right]  _{\mathbb{T}}}%
{\displaystyle\int\limits_{t_{0}}^{\delta_{+}^{T}(t_{0})}}
\left\vert A(u)Q(u,\varphi\left(  \delta_{-}(s,u)\right)  )+G\left(
u,\varphi(u),\varphi\left(  \delta_{-}(s,u)\right)  \right)  \right\vert
\Delta u\\
&  \leq r\left(  \delta_{+}^{T}(t_{0})-t_{0}\right)  \left\Vert
A(.)Q(.,\varphi\left(  \delta_{-}(s,.)\right)  )+G\left(  .,\varphi
(.),\varphi\left(  \delta_{-}(s,.)\right)  \right)  \right\Vert .
\end{align*}
This completes the proof of part (i).

To see that $C$ is continuous, suppose $\varphi$ and $\psi$ belong to $P_{T}.$
Given $\varepsilon>0$ there exists a $\delta>0$ such that $\left\Vert
\varphi-\psi\right\Vert $ $<\delta$ implies%
\begin{align*}
\left\Vert (C\varphi)(.)-(C\psi)(.)\right\Vert  &  \leq r%
{\displaystyle\int\limits_{t_{0}}^{\delta_{+}^{T}(t_{0})}}
\left[  \left\Vert A\right\Vert E_{1}\left\Vert \varphi-\psi\right\Vert
+(E_{2}+E_{3})\left\Vert \varphi-\psi\right\Vert \right]  \Delta u\\
&  \leq r\left(  \delta_{+}^{T}(t_{0})-t_{0}\right)  \left(  \left\Vert
A\right\Vert E_{1}+E_{2}+E_{3}\right)  \left\Vert \varphi-\psi\right\Vert
<\varepsilon.
\end{align*}
By choosing $\delta=\varepsilon/N,$ we prove that $C$ is continuous.

In order to show that $C$ is compact, we consider the set $D:=\left\{
\varphi\in P_{T}:\left\Vert \varphi\right\Vert \leq R\right\}  $ for a
positive fixed constant $R.$ Consider a sequence of $T$-periodic functions in
shifts, $\left\{  \varphi_{n}\right\}  ,$ and assume that $\left\{
\varphi_{n}\right\}  \in D.$ Moreover, from (\ref{5.5}) and (\ref{5.6}) we get%
\begin{align*}
\left\vert Q(t,x)\right\vert  &  =\left\vert Q(t,x)-Q(t,0)+Q(t,0)\right\vert
\\
&  \leq\left\vert Q(t,x)-Q(t,0)\right\vert +\left\vert Q(t,0)\right\vert \\
&  \leq E_{1}\left\Vert x\right\Vert +\alpha
\end{align*}
and%
\begin{align*}
\left\vert G(t,x,y)\right\vert  &  =\left\vert
G(t,x,y)-G(t,0,0)+G(t,0,0)\right\vert \\
&  \leq\left\vert G(t,x,y)-G(t,0,0)\right\vert +\left\vert G(t,0,0)\right\vert
\\
&  \leq E_{2}\left\Vert x\right\Vert +E_{3}\left\Vert y\right\Vert +\beta,
\end{align*}
where $\alpha=\left\vert Q(t,0)\right\vert $ and $\beta=\left\vert
G(t,0,0)\right\vert .$ If we consider $\left\Vert \left(  C\varphi_{n}\right)
(.)\right\Vert ,$ we have%
\[
\left\Vert \left(  C\varphi_{n}\right)  (.)\right\Vert \leq r\left(
\delta_{+}^{T}(t_{0})-t_{0}\right)  \left[  \left\Vert A\right\Vert \left(
E_{1}\left\Vert \varphi_{n}\right\Vert +\alpha\right)  +\left(  E_{2}%
+E_{3}\right)  \left\Vert \varphi_{n}\right\Vert +\beta\right]  ,
\]
where $r$ is as in (\ref{r}). Since $\left\{  \varphi_{n}\right\}  \in D,$ we
obtain $\left\Vert \left(  C\varphi_{n}\right)  (.)\right\Vert \leq L,$ where
\[
L=r\left(  \delta_{+}^{T}(t_{0})-t_{0}\right)  \left[  \left\Vert A\right\Vert
\left(  E_{1}\left\Vert \varphi_{n}\right\Vert +\alpha\right)  +\left(
E_{2}+E_{3}\right)  \left\Vert \varphi_{n}\right\Vert +\beta\right]  .
\]

Now, we evaluate $\left(  C\varphi_{n}\right)  ^{\Delta}(t)$ and show that
$\left(  C\varphi_{n}\right)  $ is uniformly bounded.%
\begin{align*}
\left(  C\varphi_{n}\right)  ^{\Delta}(t)  &  =\Phi_{A}^{\Delta}%
(t,t_{0})\left(  \Phi_{A}^{-1}(\delta_{+}^{T}(t_{0}),t_{0})-I\right)  ^{-1}\\
&  \times%
{\displaystyle\int\limits_{t}^{\delta_{+}^{T}(t)}}
\Phi_{A}^{-1}(\sigma(u),t_{0})\left[  A(u)Q(u,\varphi_{n}\left(  \delta
_{-}(s,u)\right)  )+G\left(  u,\varphi_{n}(u),\varphi_{n}\left(  \delta
_{-}(s,u)\right)  \right)  \right]  \Delta u\\
&  +\Phi_{A}(\sigma\left(  t\right)  ,t_{0})\left(  \Phi_{A}^{-1}(\delta
_{+}^{T}(t_{0}),t_{0})-I\right)  ^{-1}\\
&  \times\left[  \Phi_{A}^{-1}(\sigma(\delta_{+}^{T}(t)),t_{0})\left[
A(\delta_{+}^{T}(t))Q(\delta_{+}^{T}(t),\varphi_{n}\left(  \delta_{-}%
(s,\delta_{+}^{T}(t))\right)  )\right.  \right. \\
&  +\left.  \left.  G\left(  \delta_{+}^{T}(t),\varphi_{n}(\delta_{+}%
^{T}(t)),\varphi_{n}\left(  \delta_{-}(s,\delta_{+}^{T}(t))\right)  \right)
\right]  \delta_{+}^{\Delta T}(t)\right. \\
&  -\left.  \Phi_{A}^{-1}(\sigma(t),t_{0})\left[  A(t)Q(t,\varphi_{n}\left(
\delta_{-}(s,t)\right)  )+G\left(  t,\varphi_{n}(t),\varphi_{n}\left(
\delta_{-}(s,t)\right)  \right)  \right]  \right]  .
\end{align*}
This along with (\ref{5.2}-\ref{5.4}) and
\[
\Phi_{A}^{\Delta}(t,t_{0})=A(t)\Phi_{A}(t,t_{0}),
\]
implies%
\begin{align}
\left(  C\varphi_{n}\right)  ^{\Delta}(t)  &  =A(t)\left(  C\varphi
_{n}\right)  (t)+\Phi_{A}(\sigma\left(  t\right)  ,t_{0})\left(  \Phi_{A}%
^{-1}(\delta_{+}^{T}(t_{0}),t_{0})-I\right)  ^{-1}\nonumber\\
&  \times\left[  \left(  \Phi_{A}^{-1}(\sigma\left(  \delta_{+}^{T}(t)\right)
,t_{0})-\Phi_{A}^{-1}(\sigma(t),t_{0})\right)  \left[  A(t)Q(t,\varphi
_{n}\left(  \delta_{-}(s,t)\right)  )+G\left(  t,\varphi_{n}(t),\varphi
_{n}\left(  \delta_{-}(s,t)\right)  \right)  \right]  \right]  . \label{5.15}%
\end{align}
Substituting%
\[
\Phi_{A}^{-1}(\sigma\left(  \delta_{+}^{T}(t)\right)  ,t_{0})-\Phi_{A}%
^{-1}(\sigma(t),t_{0})=\left(  \Phi_{A}^{-1}(\delta_{+}^{T}(t_{0}%
),t_{0})-I\right)  \Phi_{A}^{-1}(\sigma\left(  t\right)  ,t_{0})
\]
in (\ref{5.15}), we obtain%
\[
\left(  C\varphi_{n}\right)  ^{\Delta}(t)=A(t)\left(  C\varphi_{n}\right)
(t)+A(t)Q(t,\varphi_{n}\left(  \delta_{-}(s,t)\right)  )+G\left(
t,\varphi_{n}(t),\varphi_{n}\left(  \delta_{-}(s,t)\right)  \right)  .
\]
Thus, $\left(  C\varphi_{n}\right)  ^{\Delta}(t)$ is bounded. This means that
$\left(  C\varphi_{n}\right)  $ is uniformly bounded and equicontinuous. Hence
by Arzela-Ascoli theorem $C(D)$ is compact.
\end{proof}

\begin{lemma}
\label{contraction}Let $B$ be given by (\ref{5.13}). If (\ref{5.5}) holds with
$E_{1}<\zeta<1$, then $B$ is a contraction.
\end{lemma}

\begin{proof}
Let $B$ be defined by (\ref{5.13}). Then for $\varphi,\psi\in P_{T},$ we have%
\begin{align*}
\left\Vert \left(  B\varphi\right)  (.)-\left(  B\psi\right)  (.)\right\Vert
&  =\max_{t\in\left[  t_{0},\delta_{+}^{T}(t_{0})\right]  _{\mathbb{T}}%
}\left\vert \left(  B\varphi\right)  (t)-\left(  B\psi\right)  (t)\right\vert
\\
&  =\max_{t\in\left[  t_{0},\delta_{+}^{T}(t_{0})\right]  _{\mathbb{T}}%
}\left\vert Q(t,\varphi\left(  \delta_{-}(s,t)\right)  )-Q(t,\psi\left(
\delta_{-}(s,t)\right)  )\right\vert \\
&  \leq E_{1}\left\Vert \varphi-\psi\right\Vert \\
&  <\zeta\left\Vert \varphi-\psi\right\Vert .
\end{align*}
This shows $B$ is a contraction mapping with contraction constant $\zeta.$
\end{proof}

\begin{theorem}
\label{alpha-beta}Assume that all hypothesis of Lemma \ref{c-c} are satisfied.
Let $r$ be given by (\ref{r}), $\alpha:=\left\Vert Q(t,0)\right\Vert $ and
$\beta:=\left\Vert G(t,0,0)\right\Vert .$ Let $J$ be a positive constant
satisfying the inequality%
\[
E_{1}J+\alpha+r\left(  \delta_{+}^{T}(t_{0})-t_{0}\right)  \left[  \left\Vert
A\right\Vert (\alpha+E_{1}J)+(E_{2}+E_{3})J+\beta\right]  \leq J.
\]
Then the equation (\ref{5.1}) has a solution in $\mathbb{M}:=\left\{
\varphi\in P_{T}:\left\Vert \varphi\right\Vert \leq J\right\}  .$
\end{theorem}

\begin{proof}
By Lemma \ref{c-c}, $C$ is continuous and compact. Also $B$ is a contraction
of $P_{T}.$ Now, we have to show that $\left\Vert B\psi+C\varphi\right\Vert
\leq J$ for $\varphi,\psi\in\mathbb{M}.$ Take $\varphi$ and $\psi$ from
$\mathbb{M}$ then%
\begin{align*}
\left\Vert B\psi(.)+C\varphi(.)\right\Vert  &  \leq E_{1}\left\Vert
\psi\right\Vert +\alpha+r%
{\displaystyle\int\limits_{t_{0}}^{\delta_{+}^{T}(t_{0})}}
\left[  \left\Vert A\right\Vert \left(  \alpha+E_{1}\left\Vert \varphi
\right\Vert \right)  +(E_{2}+E_{3})\left\Vert \varphi\right\Vert
+\beta\right]  \Delta u\\
&  \leq E_{1}J+\alpha+r\left(  \delta_{+}^{T}(t_{0})-t_{0}\right)  \left[
\left\Vert A\right\Vert (\alpha+E_{1}J)+(E_{2}+E_{3})J+\beta\right] \\
&  \leq J.
\end{align*}
By Krasnoselskii's theorem, there exists a fixed point $z\in$ $\mathbb{M}$
such that $z=Bz+Cz$. This fixed point is also a $T$-periodic solution of
(\ref{5.1}) in shifts $\delta_{\pm}$. The proof is complete.
\end{proof}

\begin{theorem}
In addition to all hypothesis of Lemma \ref{c-c} suppose also that%
\[
E_{1}+r\left(  \delta_{+}^{T}(t_{0})-t_{0}\right)  \left(  \left\Vert
A\right\Vert E_{1}+E_{2}+E_{3}\right)  <1.
\]
Then the equation (\ref{5.1}) has a unique solution which is $T$-periodic in shifts.
\end{theorem}

\begin{proof}
Let the mapping $H$ is defined by (\ref{H}) and $\varphi,\psi\in P_{T}.$ Since%
\[
\left\Vert \left(  H\varphi\right)  \left(  .\right)  -\left(  H\psi\right)
\left(  .\right)  \right\Vert \leq\left[  E_{1}+r\left(  \delta_{+}^{T}%
(t_{0})-t_{0}\right)  \left(  \left\Vert A\right\Vert E_{1}+E_{2}%
+E_{3}\right)  \right]  \left\Vert \varphi-\psi\right\Vert
\]
the proof follows from contraction mapping principle. The proof is complete.
\end{proof}

The following result is a generalization of \cite[Corollary 2.7]{4} and
\cite[Corollary 2.8]{5}.

\begin{corollary}
\label{cor.3}Assume that (\ref{5.2})-(\ref{5.4.1}) hold. Let $\alpha$ and
$\beta$ be the constants as in Theorem \ref{alpha-beta}. Suppose that there
exist positive constants $E_{1}^{\ast},E_{2}^{\ast}$ and $E_{3}^{\ast}$ such
that%
\begin{equation}
\left\vert Q(t,x)-Q(t,y)\right\vert \leq E_{1}^{\ast}\left\Vert x-y\right\Vert
, \label{E1*}%
\end{equation}%
\begin{equation}
\left\vert G(t,x,y)-G(t,z,w)\right\vert \leq E_{2}^{\ast}\left\Vert
x-z\right\Vert +E_{3}^{\ast}\left\Vert y-w\right\Vert \label{E2 *E3*}%
\end{equation}
and%
\begin{equation}
E_{1}^{\ast}J+\alpha+r\left(  \delta_{+}^{T}(t_{0})-t_{0}\right)  \left[
\left\Vert A\right\Vert (\alpha+E_{1}^{\ast}J)+(E_{2}^{\ast}+E_{3}^{\ast
})J+\beta\right]  \leq J \label{inq}%
\end{equation}
holds for all $x,y,z,$ and $w\in\mathbb{M}$. Then (\ref{5.1}) has a solution
in $\mathbb{M}$. Moreover, if%
\[
E_{1}^{\ast}+r\left(  \delta_{+}^{T}(t_{0})-t_{0}\right)  \left(  \left\Vert
A\right\Vert E_{1}^{\ast}+E_{2}^{\ast}+E_{3}^{\ast}\right)  <1,
\]
then the solution in $\mathbb{M}$ is unique.
\end{corollary}

The following example illustrates our existence results.

\begin{example}
Consider the time scale $\mathbb{T=}\left\{  2^{n}:n\in\mathbb{Z}\right\}
\cup\left\{  0\right\}  $, which is $2$-periodic in shifts $\delta_{\pm
}(s,t)=s^{\pm1}t$ associated with the initial point $t_{0}=1$. For a positive
constant $J$ define the set $\mathbb{M}_{J}$ by
\[
\mathbb{M}_{J}:=\left\{  \varphi\in P_{2}:\left\Vert \varphi\right\Vert \leq
J\right\}  \text{,}%
\]
where $P_{2}$ is the set of $2-$periodic functions in shifts $\delta_{\pm
}(s,t)=s^{\pm1}t$. Substituting%
\[
A\left(  t\right)  =\left[
\begin{array}
[c]{cc}%
\frac{1}{t} & 0\\
0 & \frac{1}{t}%
\end{array}
\right]  ,
\]%
\[
Q(t,u)=\left[
\begin{array}
[c]{c}%
\frac{1}{8}\left(  (-1)^{\frac{\ln t}{\ln\sqrt{2}}}+u\right)  \\
0
\end{array}
\right]  ,
\]
and%
\[
G(t,u,v)=\left[
\begin{array}
[c]{c}%
\frac{1}{8t}\sin\left(  \frac{\ln t}{\ln\sqrt{2}}\pi\right)  u\\
0
\end{array}
\right]  ,
\]
into (\ref{5.1}) we obtain the dynamic system
\begin{equation}
x^{\Delta}\left(  t\right)  =\left[
\begin{array}
[c]{cc}%
\frac{1}{t} & 0\\
0 & \frac{1}{t}%
\end{array}
\right]  x(t)+\left[
\begin{array}
[c]{c}%
\frac{1}{8}\left(  (-1)^{\frac{\ln t}{\ln\sqrt{2}}}+x(\delta_{-}(s,t))\right)
\\
0
\end{array}
\right]  ^{\Delta}+\left[
\begin{array}
[c]{c}%
\frac{1}{8t}\sin\left(  \frac{\ln t}{\ln\sqrt{2}}\pi\right)  x(t)\\
0
\end{array}
\right]  .\label{ex}%
\end{equation}
One may easily verify that (\ref{5.2})-(\ref{5.4.1}) hold for all $x\in P_{2}$
and all $t\in\mathbb{T}^{\ast}=\left\{  2^{n}:n\in\mathbb{Z}\right\}  $.
\newline Similar to \cite[Example 6]{2}, one may conclude that%
\[
\Phi_{A}\left(  \delta_{+}^{2}\left(  1\right)  ,1\right)  =\left[
\begin{array}
[c]{cc}%
2 & 0\\
0 & 2
\end{array}
\right]  ,
\]
which along with Theorem \ref{thm3} shows that the homogeneous system%
\[
x^{\Delta}\left(  t\right)  =\left[
\begin{array}
[c]{cc}%
\frac{1}{t} & 0\\
0 & \frac{1}{t}%
\end{array}
\right]  x(t)
\]
has no periodic solution in shifts. For $\varphi,\psi\in\mathbb{M}$, we have%
\begin{align*}
\left\vert Q(t,\varphi(\delta_{-}(s,t)))-Q(t,\psi(\delta_{-}%
(s,t)))\right\vert  &  \leq\frac{1}{8}\max_{t\in\left[  1,4\right]
_{\mathbb{T}}}\left\vert \left(  \varphi(\delta_{-}(s,t))-\psi(\delta
_{-}(s,t))\right)  \right\vert \\
&  \leq\frac{1}{8}\left\Vert \varphi-\psi\right\Vert
\end{align*}
which shows that (\ref{E1*}) holds for $E_{1}^{\ast}=\frac{1}{8}$. For
$\varphi,\psi\in\mathbb{M}$ we also have%
\begin{align*}
\left\vert G(t,\varphi(t),\varphi(\delta_{-}(s,t)))-G(t,\psi(t),\psi
(\delta_{-}(s,t)))\right\vert  &  \leq\frac{1}{8}\max_{t\in\left[  1,2\right]
_{\mathbb{T}}}\left\vert \frac{1}{t}\sin\left(  \frac{\ln t}{\ln\sqrt{2}}%
\pi\right)  \left(  \varphi(t)-\psi(t)\right)  \right\vert \\
&  \leq\frac{1}{8}\left\Vert \varphi-\psi\right\Vert
\end{align*}
Thus, (\ref{E2 *E3*}) holds for $E_{2}^{\ast}=\frac{1}{8}$ and $E_{3}^{\ast
}=0$. Hence, for such $E_{1}^{\ast}$, $E_{2}^{\ast}$, and $E_{3}^{\ast}$ the
inequality (\ref{inq}) turns into%
\begin{equation}
\frac{2}{5}\leq J\label{res.inq.}%
\end{equation}
since $\alpha=\frac{1}{8}$, $\beta=0$, $\left\Vert A\right\Vert =1,$ and%
\begin{align*}
r &  =\max_{t\in\left[  1,2\right]  _{\mathbb{T}}}\left(  \max_{u\in\left[
t,2t\right]  _{\mathbb{T}}}\left\vert \left[  \Phi_{A}(\sigma(u),t_{0})\left(
\Phi_{A}^{-1}(\delta_{+}^{T}(1),1)-I\right)  \Phi_{A}^{-1}(t,t_{0})\right]
^{-1}\right\vert \right)  \\
&  =\max_{t\in\left[  1,2\right]  _{\mathbb{T}}}\left(  \max_{u\in\left[
t,2t\right]  _{\mathbb{T}}}\left\vert \left[  -2^{\frac{\ln u}{\ln2}}\Phi
_{A}^{-1}(t,t_{0})\right]  ^{-1}\right\vert \right)  \\
&  \leq\max_{t\in\left[  1,2\right]  _{\mathbb{T}}}\left(  \left\vert
2^{-\frac{\ln t}{\ln2}}\Phi_{A}(t,t_{0})\right\vert \right)  \\
&  =\max_{t\in\left[  1,2\right]  _{\mathbb{T}}}\left(  \left\vert
2^{-\frac{\ln t}{\ln2}}tI\right\vert \right)  \\
&  =1.
\end{align*}
By Corollary \ref{cor.3} the system (\ref{ex}) has a $2$-periodic solution in
shifts. Moreover, since%
\[
E_{1}^{\ast}+r\left(  \delta_{+}^{T}(t_{0})-t_{0}\right)  \left(  \left\Vert
A\right\Vert E_{1}^{\ast}+E_{2}^{\ast}+E_{3}^{\ast}\right)  =\frac{3}{8}<1
\]
the periodic solution of the system (\ref{ex}) is unique.
\end{example}

\end{document}